\documentclass{amsart}
\usepackage[utf8]{inputenc}
\usepackage{comment}
\usepackage{amscd}
\usepackage{todonotes}
\usepackage{amssymb}
\usepackage{pstricks}
\usepackage{graphicx}
\usepackage{lineno}
\usepackage{subcaption}
\usepackage{paralist}
\usepackage[misc]{ifsym}
\usepackage{epsfig} 
\usepackage{epstopdf} 

\usepackage{mathrsfs}
\usepackage{graphics}
\usepackage{graphicx}
\usepackage{tikz-cd}

\usepackage[colorlinks=true]{hyperref}

\newtheorem{theorem}{Theorem}[section]

\newtheorem{proposition}[theorem]{Proposition}

\theoremstyle{definition}
\newtheorem{definition}[theorem]{Definition}
\newtheorem{remark}[theorem]{Remark}

\newcommand{\R}{\mathbb{R}}
\newcommand{\N}{\mathbb{N}}

\newcommand{\K}{\mathbb{K}}

\newcommand{\Ss}{\mathbb{S}}

\newcommand{\PP}{\mathcal{P}}

\title[Dynamics of 3D nilpotent vector fields]{Discrete and continuous dynamics of real $3$-dimensional nilpotent polynomial vector fields}

\author[]{\'Alvaro Casta\~neda}
\author[]{Salom\'on Rebollo--Perdomo}


\thanks{The first author is funded by FONDECYT Regular Grant 1200653. The second author is supported by Universidad de B\'{\i}o-B\'{\i}o Grant 2020134 IF/R}

\address{Universidad de Chile, Departamento de Matem\'aticas. Casilla 653, Santiago, Chile}
\email{castaneda@uchile.cl }

\address{Departamento de Matem\'atica, Facultad de Ciencias,Universidad del B\'{\i}o-B\'{\i}o, Casilla 5--C, Concepci\'on,
Chile}
\email{srebollo@ubiobio.cl }

\begin{document}

\maketitle

\begin{abstract}
    A large class of real $3$-dimensional nilpotent polynomial vector fields of arbitrary degree is considered. The aim of this work is to present general properties of the discrete and continuous dynamical systems induced by these vector fields. In the discrete case, it is proved that each dynamical system has a unique fixed point and no $2$-cycles.  Moreover, either the fixed point is a global attractor or  there exists  a $3$-cycle which is not a repeller. In the continuous setting, it is proved that each dynamical system is polynomially integrable. In addition, for a subclass of the considered vector fields, the system is polynomially completely integrable. Furthermore, for a family of low degree vector fields, it is provided a more precise description about the global dynamics of the trajectories of the induced dynamical system. In particular, it is proved the existence of an invariant surface foliated by periodic orbits.
Finally, some remarks and open questions, motivated by our results, the Markus--Yamabe Conjecture and the problem of planar limit cycles, are given.
\end{abstract}

\section{Introduction}
The study of 
\emph{nilpotent polynomial vector fields}, 
i.e. polynomial vector fields 
$\displaystyle F \colon \K^n \longrightarrow \K^n$ 
whose Jacobian matrix $JF$ is nilpotent, 
where $\K$ is a field  of characteristic zero, 
is closely related to the Jacobian Conjecture~\cite{Ess2000,EKC2021}.  
Indeed, the seminal works of 
A.V. Yagzhev~\cite{Yag} and H. Bass et al. \cite{BCW} 
prove that for analyzing the Jacobian Conjecture 
is sufficient to focus on polynomial vector 
fields of the form $I+F$, where $I$ is the identity 
and $F$ is nilpotent and homogeneous of degree three.
Another 
motivation for studying nilpotent polynomial vector fields
arises from dynamics. Recall  that
each real vector field 
$F \displaystyle \colon \R^n \longrightarrow \R^n$ 
induces 
a discrete dynamical system 
defined by the iteration of $F$
and a continuous dynamical system 
defined by the flow generated by 
the differential system associated with $F$.
For this induced continuous (resp. discrete) dynamical system  L. Markus and H. Yamabe \cite{MY} 
(resp. J. LaSalle \cite{LaSalle}) 
 established the continuous (resp. discrete) 
global stability conjecture, which is true for $n \leq  2$ with $F$ polynomial, but it admits counterexamples for $n\geq 3$. Such counterexamples have the form $\lambda I + F$ where $\lambda < 0 \, (\textrm{resp. } |\lambda| < 1)$  and $F$ a nilpotent polynomial vector field, see \cite{CEGMH,CGM}.

Therefore,  
the nilpotent polynomial vector fields 
play a fundamental role in order to give a 
negative or positive answer to the Jacobian Conjecture as well as in the construction of examples and counterexamples to the Markus--Yamabe and LaSalle conjectures.
Furthermore, the characterization and understanding of this kind of vector fields 
in any dimension and any degree, even inhomogeneous, 
go beyond these conjectures and represent 
a challenging open problem by itself.

The characterization of nilpotent polynomial 
vector fields is well-known in dimension two \cite[p. 148]{Ess2000}. In dimension three, it depends on the linear dependence of the components of $F$ over $\K$. When the components are linearly dependent, it is given in \cite[Corollary 1.1]{ChE} (the former counterexamples to both Markus--Yamabe and LaSalle conjectures have linearly dependent components).   When the components are linearly independent, the first steps towards such a characterization were taken by  M. Chamberland and A. van den Essen in \cite{ChE}. In particular, they prove in \cite[Theorem 2.1]{ChE} that any polynomial vector field
$F(x,y,z)=(u(x,y), v(x,y,z), h(u(x,y)))$ is nilpotent if and only if 
\begin{displaymath} 
 F(x,y,z) = \left(g(y+b(x)),v_1 z
- (b_1 + 2 v_1 \alpha x) g(y+b(x)),\alpha (g(y+b(x)))^2\right) ,
\end{displaymath}
where $ b(x) = b_1 x + v_1 \alpha x^2 $, $v_1 \alpha \neq 0 $ and $ g \in \K[t]$, with $\deg g \geq 1$ and $ g(0) = 0$. 
Later, 
Chamberland consider the above form of $F$
with the particular parameters: 
$b_1 = 0, v_1 = 1$ and $ \alpha = -1$. 
He showed in \cite{Cha} that 
the discrete dynamical system 
induced by such a particular vector field 
has a unique fixed point, 
there are not 2-cycles, 
and under a suitable condition on the function $g$ there exists 
a $3$-cycle, which show that the nilpotent polynomial vector fields can induce a rich dynamics.

Another step in the task of 
the characterization of 
nilpotent polynomial vector fields 
in dimension three 
has been done by 
D. Yan and G. Tang in \cite{Dan1}, 
where they generalize the results of \cite{ChE}. 
This characterization problem 
has been followed by several authors; 
see for instance \cite{Dan2} and references there in. 
One of the most recent and general results in this issue is \cite[Theorem 1]{CaVa}, which gives
the 
characterization of all nilpotent polynomial 
vector fields 
$\displaystyle F \colon \K^n \longrightarrow \K^n$  
of the form
\begin{equation*}
F(x_1,x_2,\ldots,x_n) = 
(F_1(x_1,x_2),F_2(x_1,x_2,x_3), \ldots, F_{n-1}(x_1,x_2,x_{n}),F_n(x_1,x_2)).
\end{equation*}
In the three dimensional real case and by changing 
the variables $x_1,x_2,x_3$ by $x,y,z$, such a characterization is as follows. The polynomial vector field
\begin{equation}
\label{campo-vectorial-3d-corto}
F \colon \R^3 \longrightarrow \R^3,\; (x,y,z)\longmapsto (F_1(x,y),F_2(x,y,z),F_3(x,y)),
\end{equation}
is nilpotent  if and only if
\begin{equation}
 \label{campo-vectorial-3d}
    \begin{aligned}
    F_1(x,y) &=  P_1\big(y + A_1(x)\big),\\[0.5pt]
    F_2(x,y,z) &=  P_2\Big(z + \frac{1}{d_2 p_{d_2}}                                  A_2(x)\Big)-A_1'(x)F_1(x,y),\\[0.5pt]
    F_3(x,y) &= -\frac{1}{d_{2} {p_d}_{2}} \left[-\frac{1}{2}                          {A_1''(x)} \big(F_1(x,y)\big)^{2}+A_{2}'(x)                        F_1(x,y)\right]+A_3,
    \end{aligned}
\end{equation}
where 
\begin{equation}
\label{cond-Pols-P-3d}
\left \{ \begin{aligned}
    & P_i \in \mathbb{R}[s],\; d_i:=\deg P_i \geq 1,\; 
p_{d_i}:= \mbox{the leading coefficient of $P_i$},\\
& A_1(x)=a_{10}+a_{11}x+a_{12}x^2, \,\,
A_2(x)=a_{20}+a_{21}x, \: A_3\in \R.\\
&\textrm{If} \, \,  d_2 > 1, \, \,  then \, \, A_1''(x)\equiv 0.
\end{aligned}
\right .
\end{equation}

In this work, inspired by Chamberland's article \cite{Cha}, we will analyze the discrete and continuous dynamics induced by the nilpotent polynomial vector fields \eqref{campo-vectorial-3d-corto}, whose components are in \eqref{campo-vectorial-3d}. More precisely,  on the one hand, we will study the discrete dynamical system $(\R^3,\N_0,F)$, where the dynamics is given by
\begin{equation}
\label{sistema-discreto-3d}
	\begin{aligned}
	{x}_{k+1} &= F_1\left(x_{k},y_{k}\right),  \\[0.5pt]
    {y}_{k+1} &= F_2\left(x_{k},y_{k},z_{k}\right),  \\[0.5pt]
    {z}_{k+1} &= F_3\left(x_{k},y_{k}\right).
    \end{aligned}
\end{equation}
On the other hand, the continuous dynamical system $(\R^3,\R,\Phi)$, where $\Phi$ is the flow generated by the differential system
\begin{equation}
    \label{sistema-diferencial-3d}
    \begin{aligned}
    \dot{x} &=  F_1(x,y),\\[0.5pt]
    \dot{y} &=  F_2(x,y,z),\\[0.5pt]
    \dot{z} &=  F_3(x,y).
\end{aligned}
\end{equation}

Concerning the discrete dynamics our main result is the following.
\begin{theorem}
\label{Teorema-1}
Each system \eqref{sistema-discreto-3d} 
has a unique fixed point and 
there are no $2$-cycles. 
In addition, 
\begin{enumerate}
\item if $\deg A_1(x)=1$, then the fixed point is a global attractor, which is reached from any initial point after three iterations;
\item if $\deg A_1(x)=2$, then the system has a $3$-cycle which is not a repeller.
\end{enumerate}
\end{theorem}

Although the assertions of this result are essentially the same as in the work of Chamberland \cite[Theorem 3.1]{Cha},  we emphasize that the above theorem is a generalization. Indeed,
 the family of nilpotent vector fields of the form \eqref{campo-vectorial-3d-corto} is wider than the studied in \cite{Cha}. For instance,  the polynomial $P_2(s)$ in \eqref{campo-vectorial-3d} is of arbitrary degree while in Chamberland's paper is linear.

\medskip
Regarding the continuous dynamics our main result is as follows.
\begin{theorem}
\label{Teorema-2}
Each differential system \eqref{sistema-diferencial-3d} is polynomially integrable. In addition, if $\deg A_1(x)=1$, then  differential system \eqref{sistema-diferencial-3d} is polynomially completely integrable.
\end{theorem}
This result gives valuable information to describe and comprehend the long-term behavior of the trajectories of  each differential system \eqref{sistema-diferencial-3d}. In particular, it says that the dynamics of the system occurs in the algebraic surfaces defined by the level sets of the polynomial first integral guaranteed by the theorem. Thus, the topology of these surfaces plays an important role in kind of orbits that they can supported. For instance, if they are simply connected surfaces and does not posses any singularity of the system, then they can not support periodic orbits of the differential system.  

In order to get more precise features on the dynamics of the trajectories of these continuous dynamical systems, we will study some particular cases according with the degrees of $P_1(s)$ and $P_2(s)$ in \eqref{cond-Pols-P-3d}. Thus, we have the following result.
\begin{proposition}
\label{resultados-caso-d1=d2=1}
Assume that $\deg P_1(s)= \deg P_2(s)=1$ in system \eqref{sistema-diferencial-3d}.
\begin{enumerate}
\item If $\deg A_1(x)=1$, then  each nontrivial trajectory of system \eqref{sistema-diferencial-3d} goes to infinity in forward and backward time.
\item If $\deg A_1(x)=2$ and we define $\mu:=A_3\,a_{12}\,p_{d_2}\,p_{d_1}^2$, then 
\begin{enumerate}
    \item  each trajectory of \eqref{sistema-diferencial-3d} goes to infinity in forward and backward time if $\mu>0$,
   \item there exists a unique cuspidal invariant surface $\mathcal{S}_{0}$ of \eqref{sistema-diferencial-3d} and each trajectory of \eqref{sistema-diferencial-3d} in 
   $\R^3 \backslash \mathcal{S}_{0}$ 
   goes to infinity in forward and backward time if $\mu=0$,
   \item there exists a unique isochronous  periodic surface $\mathcal{S}_{\mu}$ of \eqref{sistema-diferencial-3d} and each trajec\-to\-ry of \eqref{sistema-diferencial-3d} in 
   $\R^3 \backslash \mathcal{S}_{\mu}$  goes to infinity in forward and backward time if $\mu<0$.
\end{enumerate}
\end{enumerate}
\end{proposition}

The properties of 
statement $2)$ of this proposition
states an interesting and surprising analogy with 
the Bogdanov--Takens bifurcation.
Indeed, in such a bifurcation, 
we can choose a $1$-parameter curve in such a way 
that the corresponding system 
has no singularities for positive values of the parameter, 
a cusp singularity if the parameter is zero,
and a unique periodic orbit (limit cycle) for negative values of the parameter. See \cite[p. 324]{Kus}.

The paper is organized as follows. In Section 2 we will simplify the expression of dynamical systems \eqref{sistema-discreto-3d} and \eqref{sistema-diferencial-3d} through polynomial automorphisms. We use the simplified expressions to analyze the discrete dynamics in Section 3 and the continuous dynamics in Section 4. Some concluding remarks, questions and comments are given in Section 5.

\section{Simpler conjugated systems}
\label{Seccion-2}
The main idea to prove our results is the use of polynomial automorphisms of $\R^3$ to transform the original dynamical systems into new ones with simpler expressions. The transformed dynamical systems are analyzed easily.
Concretely, the polynomial map 
\begin{equation}
\label{cambio-variables-3d-normal}
\begin{tikzcd}
(u,v,w) \arrow{r}[]{\Psi} 
& 
\displaystyle \bigg(u, 
v - A_1(u), 
w - \frac{1}{d_2 p_{d_2}} A_2(u)\bigg)=(x,y,z)
\end{tikzcd}
\end{equation}
is a polynomial automorphism of $\displaystyle \R^3$, whose inverse is
\begin{equation}
\label{cambio-variables-3d-inversa}
\begin{tikzcd}
(x,y,z) \arrow{r}[]{\Psi^{-1}} 
& 
\displaystyle \bigg(x, 
y + A_1(x), 
z + \frac{1}{d_2 p_{d_2}} A_2(x)\bigg)=(u,v,w).
\end{tikzcd}
\end{equation}
If we define 
$\displaystyle G(u,v,w) := 
(\Psi^{-1} \circ F \circ \Psi)(u,v,w)$, 
then 
$\displaystyle (\R^n,\N_0,F)$ and 
$\displaystyle (\R^n,\N_0,G)$ 
are conjugated. 
Explicitly, 
by using equations 
\eqref{campo-vectorial-3d}, 
\eqref{cond-Pols-P-3d}, 
\eqref{cambio-variables-3d-normal}, and 
\eqref{cambio-variables-3d-inversa},  
the discrete dynamical system
\eqref{sistema-discreto-3d} is conjugated to the system
\begin{equation}
\label{sistema-discreto-3d-conjugado}
	\begin{aligned}
	{u}_{k+1} &= P_1\left(v_{k}\right),  \\
    {v}_{k+1} &= P_2\left(w_{k}\right)+a_{12}P_1\left(v_{k}\right)\Big(P_1\left(v_{k}\right)-2u_{k}\Big)+a_{10},  \\
    {w}_{k+1} &= \frac{a_{12}}{d_2 p_{d_2}} \big(P_1(v_k)\big)^2 + A_3 + \frac{a_{20}}{d_2 p_{d_2}}.
    \end{aligned}
\end{equation}

Analogously, by using \eqref{cambio-variables-3d-inversa}, as a
change of coordinates, together with equations \eqref{campo-vectorial-3d} and \eqref{cond-Pols-P-3d}, the differential system \eqref{sistema-diferencial-3d} becomes
\begin{equation}
\label{sistema-diferencial-transformado-3d}
\begin{aligned}
    \dot{u} &= P_1(v),\\[0.5pt]
    \dot{v} &= P_2(w),\\[0.5pt]
    \dot{w} &= \frac{a_{12}}{d_{2} {p_d}_{2}}  \big(P_1(v)\big)^2+A_3.
\end{aligned}
\end{equation}
Thus, the  continuous dynamical systems associated with differential systems $\eqref{sistema-diferencial-3d}$ and \eqref{sistema-diferencial-transformado-3d} are conjugated.

\section{Discrete dynamics} 
In this section, we will prove the general properties of the discrete dynamical system \eqref{sistema-discreto-3d} stated in Theorem~\ref{Teorema-1}.

\begin{proof}[Proof of Theorem \ref{Teorema-1}]
From previous section, 
we know that the discrete dynamical systems 
\eqref{sistema-discreto-3d} 
and 
\eqref{sistema-discreto-3d-conjugado} 
are conjugated. 
Hence, we will use system \eqref{sistema-discreto-3d-conjugado} 
to give the proof of the theorem. The general part of the  result will be proved by considering two cases: $d_2>1$ and $d_2=1$. In addition, the proof of Statements $1)$ and $2)$ will be provided in the first and second cases, respectively.

{\it Case 1: $d_2>1$}. Taking in account \eqref{cond-Pols-P-3d}, $a_{12}=0$. So system
\eqref{sistema-discreto-3d-conjugado} is simplified. Thus, its fixed points
are the solutions to the algebraic system
$$
	u= P_1\left(v\right),  \quad
    v = P_2\left(w\right)+a_{10},  \quad
    w = A_3 + \frac{a_{20}}{d_2 p_{d_2}},
$$
which has the unique solution
$$
(u_0,v_0,w_0)=\left(P_1\left(v_0\right),P_2\left(w_0\right)+a_{10},A_3 + \frac{a_{20}}{d_2 p_{d_2}}\right).
$$
Moreover, if $(u,v,w)$ was part of a $2$-cycle, then 
$$	
u = P_1\big(P_2\left(w\right)+a_{10}\big),  \quad   
v = P_2\Big(A_3 + \frac{a_{20}}{d_2 p_{d_2}}\Big)+a_{10},  \quad    
{w} = A_3 + \frac{a_{20}}{d_2 p_{d_2}}.
$$
Hence $(u,v,w)=(u_0,v_0,w_0)$, the fixed point, so there are no $2$-cycles in this case. Even more, the third iteration $(u_3,v_3,w_3)$ of an arbitrary point $(u,v,w)$ is
$$
(u_{3},v_3,w_3)= 
\bigg(
P_1\Big(P_2\Big(A_3 +\frac{a_{20}}{d_2 p_{d_2}}\Big)+a_{10}\Big),
P_2\Big(A_3 +\frac{a_{20}}{d_2 p_{d_2}}\Big)+a_{10},  
A_3 +\frac{a_{20}}{d_2 p_{d_2}}
\bigg),
$$
which clearly does not depend on 
the coordinates of the initial point $(u,v,w)$.
The point $(u_3,v_3,w_3)$ is precisely 
the fixed point $(u_0,v_0,w_0)$. 
This last argument is the proof of Statement 1) because $a_{21}=0$ is equivalent to $\deg A_1(x)=1$.

{\it Case 2: $d_2=1$}. We assume  
$P_2(s)=p_{21}s+p_{20}$, 
and define $\alpha:=p_{20}+a_{10}$. Moreover, we can assume $\deg A_1(x)=2$ because otherwise we are in the previous paragraph. The polynomial map
\begin{equation*}
\begin{tikzcd}
(X,Y,Z) \arrow{r}[]{\Psi_2} 
& 
\displaystyle 
\bigg(
X+P_1(Y),\, 
Y, \,
\frac{Z+a_{12}P_1(Y)(P_1(Y)+2X)-\alpha}{p_{21}} 
\bigg)
=(u,v,w)
\end{tikzcd}
\end{equation*}
is a polynomial automorphism of $\R^3$ that gives
\begin{displaymath}
(\Psi_2^{-1}\circ G \circ \Psi_2)(X,Y,Z)=
\Big(P_1(Y)-P_1(Z),\, Z,\, a_{12}\big(P_1(Y)-P_1(Z)\big)^2+\nu\Big),
\end{displaymath}
where $\nu:=p_{21}A_3+a_{20}+\alpha$. Thus,
system \eqref{sistema-discreto-3d} is also conjugated to the system
\begin{equation}
\label{sistema-discreto-3d-conjugado-2}
	\begin{aligned}
	{X}_{k+1} &= P_1(Y_k)-P_1(Z_k),  \\[0.2pc]
    {Y}_{k+1} &= Z_k,  \\[0.0pc]
    {Z}_{k+1} &= a_{12}\big(P_1(Y_k)-P_1(Z_k)\big)^2+\nu.
    \end{aligned}
\end{equation}
From second equation in \eqref{sistema-discreto-3d-conjugado-2} it follows that this last discrete dynamical system has a unique fixed point at $(X_0,Y_0,Z_0)=(0,\nu,\nu)$. Moreover, if $(X,Y,Z)$ was part of a $2$-cycle, then 
$$
	{X} = P_1(Z)-P_1(Y),  \quad
    {Y} = a_{12}X^2+\nu,  \quad
    {Z} = a_{12}X^2+\nu.
$$
Hence $(X,Y,Z)=(0,\nu,\nu)$, so there are no $2$-cycles. This completes the proof of the general part of the theorem. To finish, we will prove Statement~2).

We claim that if the equation
\begin{equation}
\label{condicion-de-3-ciclo}
a_{12}\big(P_1(s)-P_1(\nu)\big)^2+\nu=s
\end{equation}
has a real solution $s_0\neq \nu$, then the point $(0,s_0,\nu)$ is part of a $3$-cycle of the discrete dynamical system
\eqref{sistema-discreto-3d-conjugado-2}. Indeed, if \eqref{condicion-de-3-ciclo} holds for $s_0\neq \nu$, then
$$
(0,s_0,\nu) \longrightarrow 
\big(P_1(s_0)-P_1(\nu),\nu,s_0)\big) \longrightarrow 
\big(P_1(\nu)-P_1(s_0),s_0,s_0)\big) \longrightarrow
(0,s_0,\nu),
$$
which is a $3$-cycle because no one of these points is the fixed point.
Look for a solution of equation \eqref{condicion-de-3-ciclo} is equivalent to look for a zero of the polynomial
\begin{equation}\label{h}
h(s)=a_{12}\big(P_1(s)-P_1(\nu)\big)^2+\nu-s.
\end{equation}
Since $h(s)$ has even degree and it has a simple zero at $s=\nu$ because  $h(\nu) = 0$ and $h'(\nu) = -1$,  it must have another real zero. Therefore, our claim follows. 

The linearization of 
\eqref{sistema-discreto-3d-conjugado-2}
at an arbitrary point $(X,Y,Z)$ is
the matrix
$$
\begin{pmatrix}
0 & P_1'(Y) & -P_1'(Z) \\ 
0 & 0 & 1 \\ 
0 & 2\,a_{12}(P_1(Y)-P_1(Z))P_1'(Y) & - 2\,a_{12}(P_1(Y)-P_1(Z))P_1'(Z)
\end{pmatrix}.
$$
By evaluating this matrix at each one of the three points of the $3$-cycle and computing their product we obtain the linearization of the third iteration of
\eqref{sistema-discreto-3d-conjugado-2} 
at the point $(0,s_0,\nu),$ which after using that $(P_1(s_0)-P_1(\nu))^2 = (s_0 - \nu)/a_{12}$ can be written as
$$
\begin{pmatrix}
0 & *  & *   \\ 
0 & L_{22} & *
\\ 
0 & 0 & 0
\end{pmatrix}
$$
where 
$L_{22} := 4 a_{12}^2 \big( P_{1}(s_{0}) - P_{1}(\nu)  \big)^2  (P_1^{'}(s_0))^2.$
Therefore, the $3$-cycle is an attractor when $|L_{22}|<1$ and it is a saddle when $|L_{22}|>1$. 
\end{proof}

\begin{remark}
When $P_1(s)$ is linear the equation \eqref{h} has only one solution $s_0$, different from $\nu$, and $L_{22} = 4.$ Hence, in this case the $3$-cycle of \eqref{sistema-discreto-3d-conjugado-2}  is not an attractor.
We can find suitable $a_{12}, \nu$ and $P_1(s)$ of degree two such that \eqref{sistema-discreto-3d-conjugado-2} has a $3$-cycle which is attractor.
\end{remark}

\begin{remark}
When $P_1(s)$ is linear, 
we compute, by using computational software, the Gr\"{o}ebner basis of the components of $(\Psi_2^{-1}\circ G \circ \Psi_2)^5(X,Y,Z)-(X,Y,Z)$. We obtain three polynomials, one of them depends only on $Z$ and has even degree with $\nu$ as a solution, the other two are linear in $X$ and $Y$. Hence,  \eqref{sistema-discreto-3d-conjugado-2} has a $5$-cycle. 
\end{remark}


\section{Continuous dynamics}
In this section, 
we will prove the general properties of 
the continuous dynamical system 
\eqref{sistema-diferencial-3d}.  
Before that, 
we need to recall 
some concepts related to the assertions of 
Theorem~\ref{Teorema-2} and 
Proposition~\ref{resultados-caso-d1=d2=1}.

A non-constant function 
$H \colon \R^3 \longrightarrow \R$ is a $C^r$ \emph{first integral} for differential system \eqref{sistema-diferencial-3d}
if the equation
\begin{equation*}
   \langle F, \nabla H\rangle =
   F_1H_x+F_2H_y+F_3H_z=0
\end{equation*}
holds on the whole $ \mathbb{R}^3$ and $H$ is of class $C^r$, with $r=1,2,\ldots,\infty,\omega$, where $C^{\omega}$ stands for analytic functions. In addition, if $H$ is a polynomial function, then we have a \emph{polynomial first integral}. Two $C^r$ functions $H_1(x,y,z)$  and $H_2(x,y,z)$ 
are \emph{functionally independent} in $\R^3$ if their gradients, $\nabla H_1$ and $\nabla H_2$, are linearly independent in a full Lebesgue measure subset of $\R^3$. Then, by definition, differential system  \eqref{sistema-diferencial-3d} is $C^r$(\emph{polynomially}) \emph{integrable} if it has a $C^r$ (polynomial) first integral in $\R^3$. Furthermore, it is $C^r$ (\emph{polynomially}) \emph{completely integrable} if it has two functionally independent $C^r$ (polynomial) first integrals in $\R^3$.

\begin{definition}
\label{iso}
A \emph{periodic surface of system \eqref{sistema-diferencial-3d}} 
is a surface $S\subset \R^3$, 
which is foliated by periodic orbits 
of the system. 
A periodic surface $S$ of  system \eqref{sistema-diferencial-3d}  is \emph{isochronous} when all its periodic orbits have the same period.
\end{definition}

\begin{proof}[Proof of Theorem 2]
From Section \ref{Seccion-2}, we know that differential systems \eqref{sistema-diferencial-3d} and \eqref{sistema-diferencial-transformado-3d} are polynomially conjugated through the change of coordinates \eqref{cambio-variables-3d-inversa}.
Moreover, the last two equations in \eqref{sistema-diferencial-transformado-3d} form a planar Hamiltonian system, whose Hamiltonian function is
$$
G(v,w):=\int P_2(w)\, dw - 
\frac{a_{12}}{d_{2} {p_d}_{2}}\int(P_1(v))^2\, dv
-A_3v.
$$
Then, by extending this function to $\R^3$, that is, 
 by defining the polynomial function 
\begin{equation}
\label{integral-primera-general}
H(u,v,w):=\int P_2(w)\, dw - 
\frac{a_{12}}{d_{2} {p_d}_{2}}\int(P_1(v))^2\, dv
-A_3v,
\end{equation}
we have
$$
H_{u}=0,\quad H_{v}=-\frac{a_{12}}{d_{2} {p_d}_{2}}(P_1(v))^2-A_3\quad \mbox{and}\quad H_{w}=P_2(w).
$$
Thus,
$$
P_1(v)\, H_{u}+P_2(w)\,H_{v}+\left( \frac{a_{12}}{d_{2} {p_d}_{2}}  \big(P_1(v)\big)^2+A_3\right) H_{w}= 0, \quad \forall \; (u,v,w)\in \R^3. 
$$
Hence, $H$ is a polynomial first integral of system \eqref{sistema-diferencial-transformado-3d}. Since the change of variables \eqref{cambio-variables-3d-inversa} is polynomial, also differential system \eqref{sistema-diferencial-3d} has polynomial first integral.

We now prove the second part of the theorem. Since $\deg A_1(x)=1$, $a_{12}=0$. Then, system \eqref{sistema-diferencial-transformado-3d} reduces to
\begin{equation}
\label{sistema-diferencial-transformado-3d-a12-cero}
\begin{aligned}
    \dot{u} &= P_1(v),\\[0.5pt]
    \dot{v} &= P_2(w),\\[0.5pt]
    \dot{w} &= A_3.
\end{aligned}
\end{equation}

We have proved that system \eqref{sistema-diferencial-transformado-3d} has a polynomial first integral, then we will show the existence of an additional polynomial first integral of the system.

If $A_3=0$, then \eqref{sistema-diferencial-transformado-3d-a12-cero} admits the two functionally independent polynomial first integrals
$$
H_1(u,v,w)=w 
\quad
\mbox{and}
\quad
H_2(u,v,w)=\int P_1(v)\,dv-uP_2(w).
$$
If $A_3 \neq 0$, then  \eqref{sistema-diferencial-transformado-3d-a12-cero} admits the two functionally independent polynomial first integrals
$$
H_1(u,v,w)=\int P_2(w)\,dw-A_3v
$$
and
$$
H_2(u,v,w)=A_3^{d_1+1}u-\sum_{j=0}^{d_1}(-1)^jA_3^{d_1-j}\left(\frac{d^j}{dv^j} P_1(v)\right)\,\xi_j(w),
$$
where $\xi_0(w)=w$ and $\xi_j(w)=\int P_2(w)\xi_{j-1}(w)\,dw$ for $j=1,2,\ldots,d_1$. In both previous cases $H_1(u,v,w)$ is the reduction of the polynomial first integral \eqref{integral-primera-general}.
Thus, system \eqref{sistema-diferencial-transformado-3d-a12-cero} is polynomially completely integrable. Therefore,  system \eqref{sistema-diferencial-3d}, with $\deg A_1(x)=1$,  is also polynomially completely integrable because it and system \eqref{sistema-diferencial-transformado-3d-a12-cero} are equivalent after a polynomial change of coordinates.
\end{proof}


\subsection{Case $d_1=d_2=1$} 

\begin{proof}[Proof of Proposition~\ref{resultados-caso-d1=d2=1}]
Recall that differential systems \eqref{sistema-diferencial-3d}  and \eqref{sistema-diferencial-transformado-3d} are equivalent under the polynomial change of coordinates
\eqref{cambio-variables-3d-inversa}. 

\smallskip
Statement $1)$. Since $\deg A_1(x)=1$, $a_{12}=0$. The linear change of coordinates
$$
X=\frac{1}{p_{d_1}p_{d_2}}\,u,\; 
Y=\frac{1}{p_{d_1}p_{d_2}}\,P_1(v),\; 
Z=\frac{1}{p_{d_2}}\,P_2(w)
$$
transforms the differential system \eqref{sistema-diferencial-transformado-3d}, with $a_{12}=0$, into the differential system
\begin{equation*}
\begin{aligned}
    \dot{X} &= Y,\\[0.5pt]
    \dot{Y} &= Z,\\[0.5pt]
    \dot{Z} &= A_3,
\end{aligned}
\end{equation*}
which can be solved explicitly. Indeed, the trajectory  $\phi_t(X_0,Y_0,Z_0)$ of the system passing through the point $(X_0,Y_0,Z_0)$ has the components:
$$
X(t)=\dfrac{A_{{3}}}{6}\,{t}^{3}+\frac{Z_{{0}}}{2}\,{t}^{2}+Y_{{0}}\,t+X_{{0}},\quad
Y(t)=\dfrac{A_{{3}}}{2}\,{t}^{2}+Z_{{0}}\,t+Y_{{0}},\quad
Z(t)=A_{{3}}\,{t}+Z_{{0}}.
$$
So if $(X_0,Y_0,Z_0)$ is not a singularity, then the nontrivial trajectory $\phi_t(X_0,Y_0,Z_0)$ escapes to infinity in forward and backward  time. 

\smallskip
Statement $2)$. Since $\deg A_1(x)=2$, $a_{12} \neq 0$. The linear change of coordinates
$$
X=(a_{12}\,p_{d_1})\,u,\; 
Y=(a_{12}\,p_{d_1})\,P_1(v),\; 
Z=(a_{12}\,p_{d_1}^2)\,P_2(w)
$$
transforms the differential system \eqref{sistema-diferencial-transformado-3d}, with $a_{12} \neq 0$, into the differential system
\begin{equation}
\label{sistema-diferencial-transformado-3d-lineal-a12-nocero}
\begin{aligned}
    \dot{X} &= Y,\\[0.5pt]
    \dot{Y} &= Z,\\[0.5pt]
    \dot{Z} &= Y^2+\mu,
\end{aligned}
\end{equation}
where $\mu=A_3\,a_{12}\,p_{d_2}\,p_{d_1}^2$ (as it is defined in Proposition \ref{resultados-caso-d1=d2=1}). Moreover, the first integral \eqref{integral-primera-general} for system \eqref{sistema-diferencial-transformado-3d}  becomes
$$
H(X,Y,Z)=-\mu\,Y+\frac{Z^2}{2}-\frac{Y^3}{3},
$$
which is a first integral for system \eqref{sistema-diferencial-transformado-3d-lineal-a12-nocero}.
Thus, a trajectory of the system \eqref{sistema-diferencial-transformado-3d-lineal-a12-nocero} is contained in a level surface $H^{-1}(c)\subset \R^3$ of $H$, with $c \in \R$. Since $H$ does not depend on $X$,  $H^{-1}(c)$ has the form
$$
H^{-1}(c)=\R \times G^{-1}(c),
$$ 
where $G(Y, Z)=-\mu\,Y+Z^2/2-Y^3/3$. Moreover,
the last two equations in \eqref{sistema-diferencial-transformado-3d-lineal-a12-nocero} form the planar Hamiltonian system associated with $G(Y, Z)$. 
We will give the proof of $a)$, $b)$ and $c)$ of the statement in three cases: $\mu>0$, $\mu=0$ and $\mu<0$.

\smallskip
{\it Case 1: $\mu>0$}. $G(Y, Z)$ does not have any singular point in the $YZ$-plane. Thus, $G^{-1}(c)$ is homeomorphic to $\R$ for any $c\in \R$. In addition, system \eqref{sistema-diferencial-transformado-3d-lineal-a12-nocero} does not have singularities in the whole space $\R^3$, then each
$H^{-1}(c)$ is a simply connected surface without any singularity of the system. Therefore, each trajectory goes to infinity in forward and backward time.

\smallskip
{\it Case 2: $\mu=0$}. $G(Y, Z)$ has the origin as the unique singularity in the $YZ$-plane. In fact, $(0,0)$ is a cusp singularity of $G(Y, Z)$. Since $G(0,0)=0$, $G^{-1}(0)$ is the cuspidal cubic curve. Hence, $G^{-1}(c)$ is homeomorphic to $\R$ for any $c\neq 0$. In addition, since all the singularities of \eqref{sistema-diferencial-transformado-3d-lineal-a12-nocero} are of the form $(X,0,0)$, they are contained in the cuspidal invariant (singular) surface $S_{0}:=H^{-1}(0)=\R \times G^{-1}(0)$. This implies that $H^{-1}(c)$, with $c \neq 0$ is a simply connected surface without any singularity of the system. Hence, all trajectories in $\R^3 \backslash S_{0}$ have to escape to infinity in  forward and backward time. 

\smallskip
{\it Case 3: $\mu<0$}. We can change the parameter $\mu$ by $-\beta^2$, with $\beta>0$. Then, by using the linear the change of coordinates
$
{x}=\sqrt{\beta}X,\;
{y}=\beta(Y-1),\;
{z}={\beta}^{3/2}\,Z
$
and the linear change of time $\tau=\sqrt{\beta}\, t$, the differential system \eqref{sistema-diferencial-transformado-3d-lineal-a12-nocero}, with $\mu=-\beta^2$, is transformed into the  differential system 
\begin{equation}
\label{sistema-diferencial-transformado-3d-lineal-a12-nocero-redu3}
\begin{aligned}
    x' &= y-1,\\[0.5pt]
    y' &= z,\\[0.5pt]
    z' &= y(y-2),
\end{aligned}
\end{equation}
where the prime denotes the derivative with respect to a new time variable $\tau$. Thus, for completing the proof of this case,
we will demonstrate that system \eqref{sistema-diferencial-transformado-3d-lineal-a12-nocero-redu3}
has a unique isochronous periodic surface
$\mathcal{S}^{*}$ and that all its trajectories 
in $\R^3 \backslash \mathcal{S}^{*}$ go to infinity
in forward and backward time.

The differential system
\eqref{sistema-diferencial-transformado-3d-lineal-a12-nocero-redu3}
does not have any singularity in the whole $\R^3$ and it
has the polynomial first integral
$$
H(x,y,z)=(6y^2+3z^2-2y^3)/6.
$$ 
Since this first integral  does not depend on $x$, 
$
H^{-1}(c)=\R \times G^{-1}(c),
$
where $G(y,z)=(6y^2+3z^2-2y^3)/6$. 
The last two equations in 
\eqref{sistema-diferencial-transformado-3d-lineal-a12-nocero-redu3} 
form, in the $yz$-plane, the 
planar Hamiltonian system associated with $G(y,z)$, 
whose singularities are $(0,0)$ and $(2,0)$. 
A simple computation shows 
that they are a center and a saddle, respectively.  
Thus, this Hamiltonian system  has a period 
annulus $\mathscr{P}$ surrounding the center $(0,0)$ and  
bounded by the homoclinic loop $\Gamma$
that joins the stable and the unstable 
manifolds of the saddle point $(2,0)$.
Since $G(0,0)=0$ and $G(2,0)=4/3$,
for all $c \in (0,4/3)$ the level curve $G^{-1}(c)$ 
has a connected component $\gamma_c$ homeomorphic to the unit 
circle $\Ss^1$ that forms part of $\mathscr{P}$ and 
the level surface $H^{-1}(c)$ has a connected component $
\mathcal{S}_c$ homeomorphic to the cylinder $\R \times 
\Ss^1$. See  Figure \ref{Fig1}. 

The straight lines 
$L_{0}:=\R \times \{(0,0)\}$
and
$L_{2}:=\R \times \{(2,0)\}$
are invariant by the flow of 
\eqref{sistema-diferencial-transformado-3d-lineal-a12-nocero-redu3}. 
Thus, as trajectories, 
they go to infinity in forward and backward time. 
Moreover, a straightforward analysis
on the topology of $G^{-1}(c)$ 
implies that for any $c\in \R$,
$$
H^{-1}(c)\cap \left(\R^3 \backslash \big(\cup_{c \in (0,4/3)} \mathcal{S}_c \cup L_0\cup L_2\big)\right)
$$
is formed only by disjoint simply connected surfaces. Hence,
 $i)$ only the invariant surfaces $\mathcal{S}_c$, 
with $c\in(0,4/3)$,
could support periodic orbits and $ii)$
any trajectory of system \eqref{sistema-diferencial-transformado-3d-lineal-a12-nocero-redu3} in  $\displaystyle \R^3 \backslash \cup_{c \in (0,4/3)} \mathcal{S}_c$ goes to infinity in forward and backward time. It remains to prove the existence of only one surface $S^{*}=\mathcal{S}_{c^{*}}$, with $c^{*} \in (0,4/3)$, that is foliated by periodic orbits of the same period.
\begin{center}
\begin{figure}[ht]
\begin{tabular}{ccc}
\includegraphics[scale=0.3]{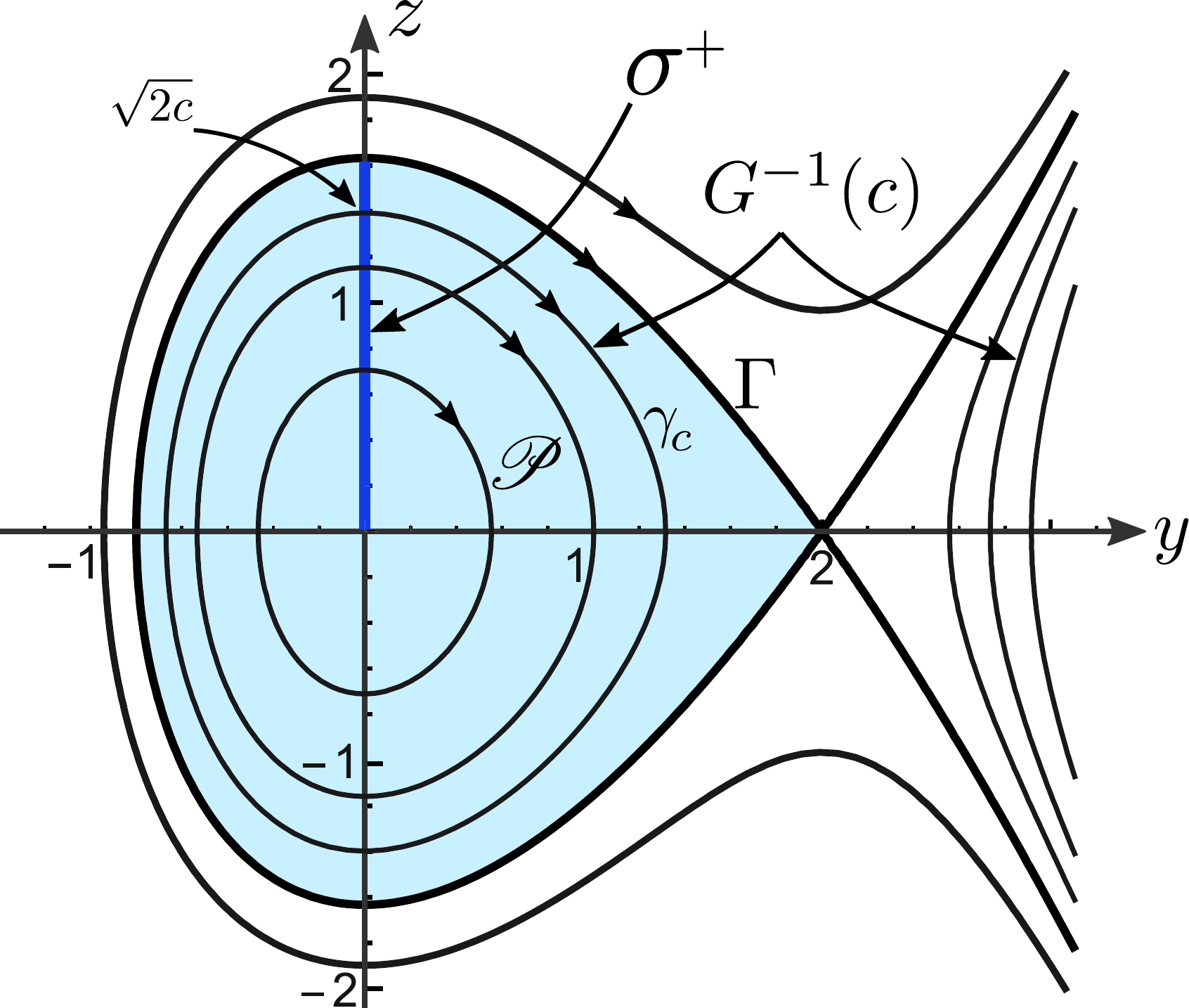}
& \hspace{0.4cm} &
\includegraphics[scale=0.32]{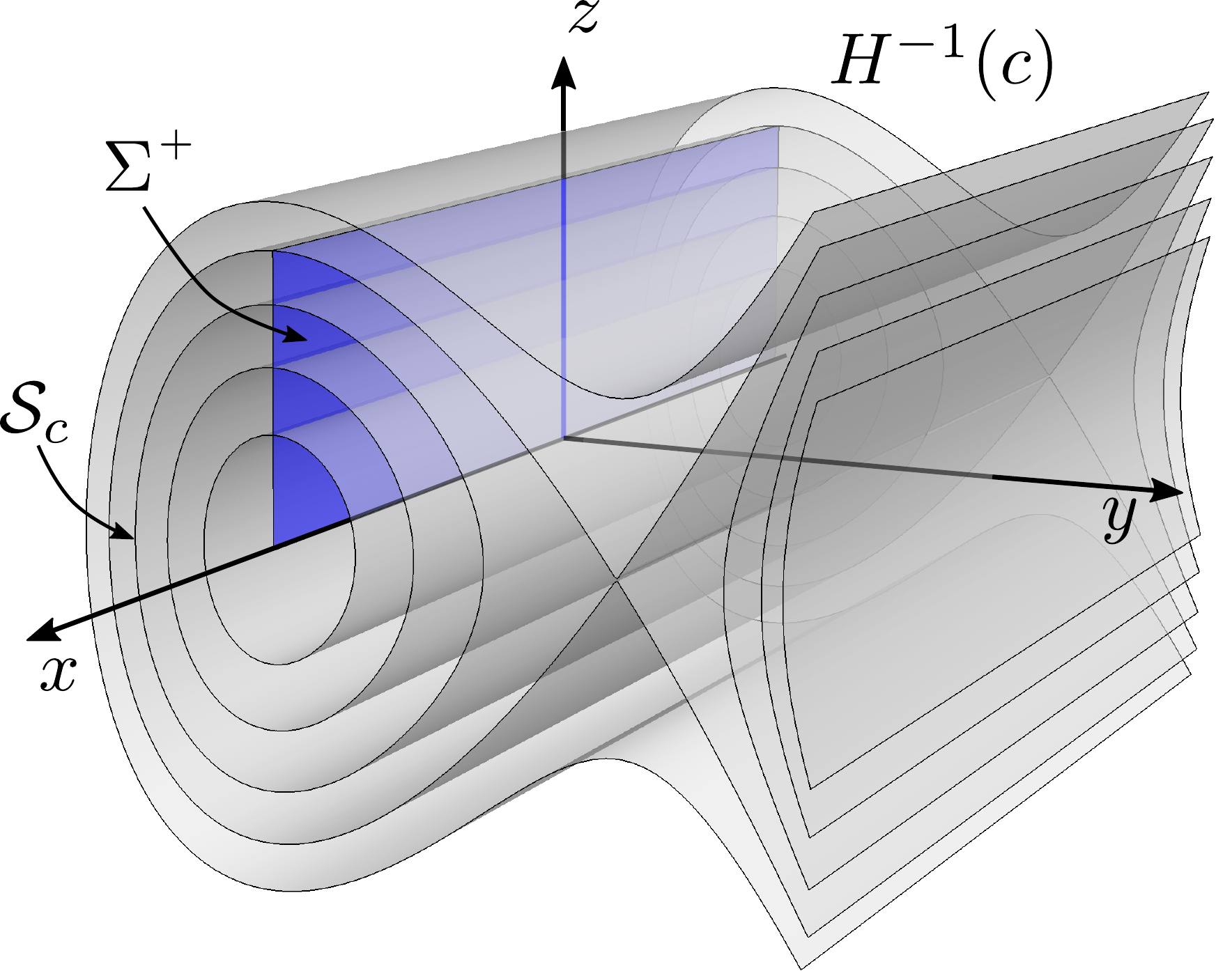}\\
$a)$ & \hspace{0.8cm} & $b)$
\end{tabular}
\caption{\label{Fig1} $a)$ Phase portrait of 
the planar Hamiltonian system associated with 
$G(y, z)$.  $b)$ Foliation of the first integral 
of \eqref{sistema-diferencial-transformado-3d-lineal-a12-nocero-redu3}. }
\end{figure}
\end{center}
In the $yz-$plane, the intersection of the period annulus $\mathscr{P}$ with the positive $z$-axis is the line segment 
$\sigma^{+}:=\{(0,z) \; | \; 0<z<\sqrt{8/3}\}$,
which is a transversal section for the flow of the
Hamiltonian system associated with $G(y, z)$.
The dot product of the vector $(0,1,0)$, which is orthogonal to $\Sigma^{+}:=\R \times \sigma^{+}$, and the vector field $\mathcal{X}$, associated to  \eqref{sistema-diferencial-transformado-3d-lineal-a12-nocero-redu3}, has defined sing: $\langle (0,1,0),\mathcal{X}\rangle=z>0$. Hence, $\Sigma^{+}$ is a 
$2$-dimensional transversal section for the flow of system~\eqref{sistema-diferencial-transformado-3d-lineal-a12-nocero-redu3}. 

As usual, we can use the energy level $c$ of $H$ to get the parametrization
$$
\mathbb{R} \times (0, 4/3) \longrightarrow \Sigma^+, \,  (x,c) \longmapsto (x,0, \sqrt{2c}),
$$
of the transversal section $\Sigma^{+}.$ In other words, the points in $\Sigma^{+}$ can be described by the two coordinates $(x,c)$.
Let $\phi_{\tau}(x,c)$ be the trajectory of system \eqref{sistema-diferencial-transformado-3d-lineal-a12-nocero-redu3} passing through $(x,c)\in \Sigma^{+}$. Since the right-hand side of the system does not depend of $x$, $\phi_{\tau}(x,c)$  has the form
$
\big(x_{c}(\tau),\varphi_{\tau}(0,c)\big),
$
where $\varphi_{\tau}(0,c)=(y_{c}(\tau),z_{c}(\tau))$  is the trajectory of the Hamiltonian system associated with $G(y,z)$, passing through the point $(0,c)\in \sigma^{+}$ at time $\tau=0$. Thus, there exists a 
well-defined Poincar\'e first return map 
$$
\begin{array}{rcl}
\PP \colon \Sigma^{+} & \longrightarrow  & \Sigma^{+}\\
(x,c) & \longmapsto &\phi_{\tau(x,c)}(x,c),
\end{array}
$$ 
where $\tau(x,c)$ is the time of first return 
of the point $(x,c)$ to $\Sigma^{+}$.

Each trajectory $\phi_{\tau}(x,c)$ 
of the system starting in the region 
$\R \times \mathscr{P} \subset \mathbb{R}^3$ is contained 
in the  surface $\mathcal{S}_c$ and 
$\Sigma^{+}\cap \mathcal{S}_c=
\{(x,c) \;|\; x \in \mathbb{R}\}$, then 
the $c$-coordinate of $\PP(x,c)$ remains 
invariant. Thus, 
$\PP(x,c)=\phi_{\tau(x,c)}(x,c)=(x_c(\tau(x,c)),c)$, 
which implies that 
the fixed points of $\PP$ are in correspondence 
with the zeros of the displacement function
$$
L(x,c):=x_c(\tau(x,c))-x_c(0).
$$

Since the right-hand side of the system (\ref{sistema-diferencial-transformado-3d-lineal-a12-nocero-redu3}) 
does not depend on $x$, the time of first 
return $\tau(x,c)$ does not either, that is, 
$\tau(x,c)=\tau(0,c)$. Thus, if $L(0,c^{*})=0$, 
then $L(x,c^*)=0$ for all $x\in \R$, whence  $\mathcal{S}_{c^{*}}$ will be a isochronous (periodic) surface, 
according to Definition \ref{iso}. 
Hence, it is enough to study the function 
$$
L(0,c)=x_c(\tau(0,c))-x_c(0), \quad \mbox{with } x_c(0)=0.
$$
To complete the proof, we will prove that
$L(0,c)<0$ for $0<c \leq 2/3$,
$L(0,c)>0$ for $2/3 \ll c <4/3$,
and $L(0,c)$ is a monotonous increasing function in $\big(2/3,4/3\big)$,
which implies the existence of a unique 
$c^{*}\in (0, 4/3)$ such that 
$L(0,c^{*})=0$. This will prove the uniqueness of the isochronous surface $\mathcal{S}_{c^{*}}$ . The proof of these assertions is analogous to the proof of the uniqueness of the limit cycle in the van der Pol differential system given in \cite[Sec 12.3]{HSD}. Hence, we will give the main ideas to prove the properties of $L(0,c)$ and we leave the details to the reader.

From the fundamental theorem of calculus and the first equation in \eqref{sistema-diferencial-transformado-3d-lineal-a12-nocero-redu3} we have
$$
L(0,c)=x_c(\tau(0,c))-x_c(0)=\int_{0}^{\tau(0,c)}x_c'(\tau)\,d \tau=
\int_{0}^{\tau(0,c)}(y_c(\tau)-1)\,d \tau.
$$
In Figure \ref{Fig2} we show an sketch for the graph 
of $y_c(\tau)-1$, whose shape follows easily from $a)$ in Figure 
\ref{Fig1} and the fact that $G(1,0) = 2/3$. Hence, for $0<c \leq 2/3$, $y_c(\tau)-1<0$
for almost all $\tau$, thus
$
L(0,c)<0
$.
\begin{center}
\begin{figure}[h]
\begin{tabular}{ccc}
\includegraphics[scale=0.31]{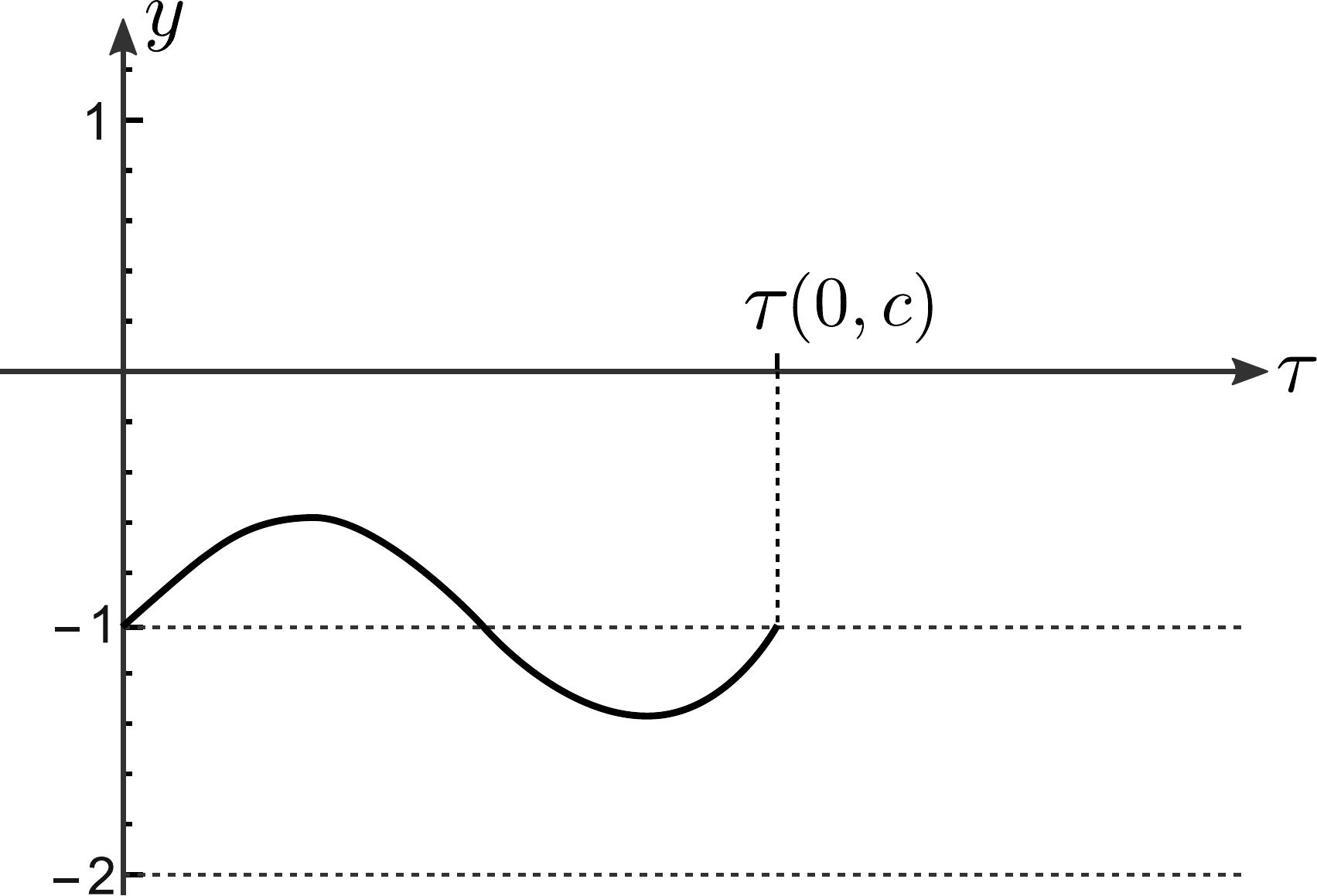}
& \hspace{0.0cm} &
\includegraphics[scale=0.31]{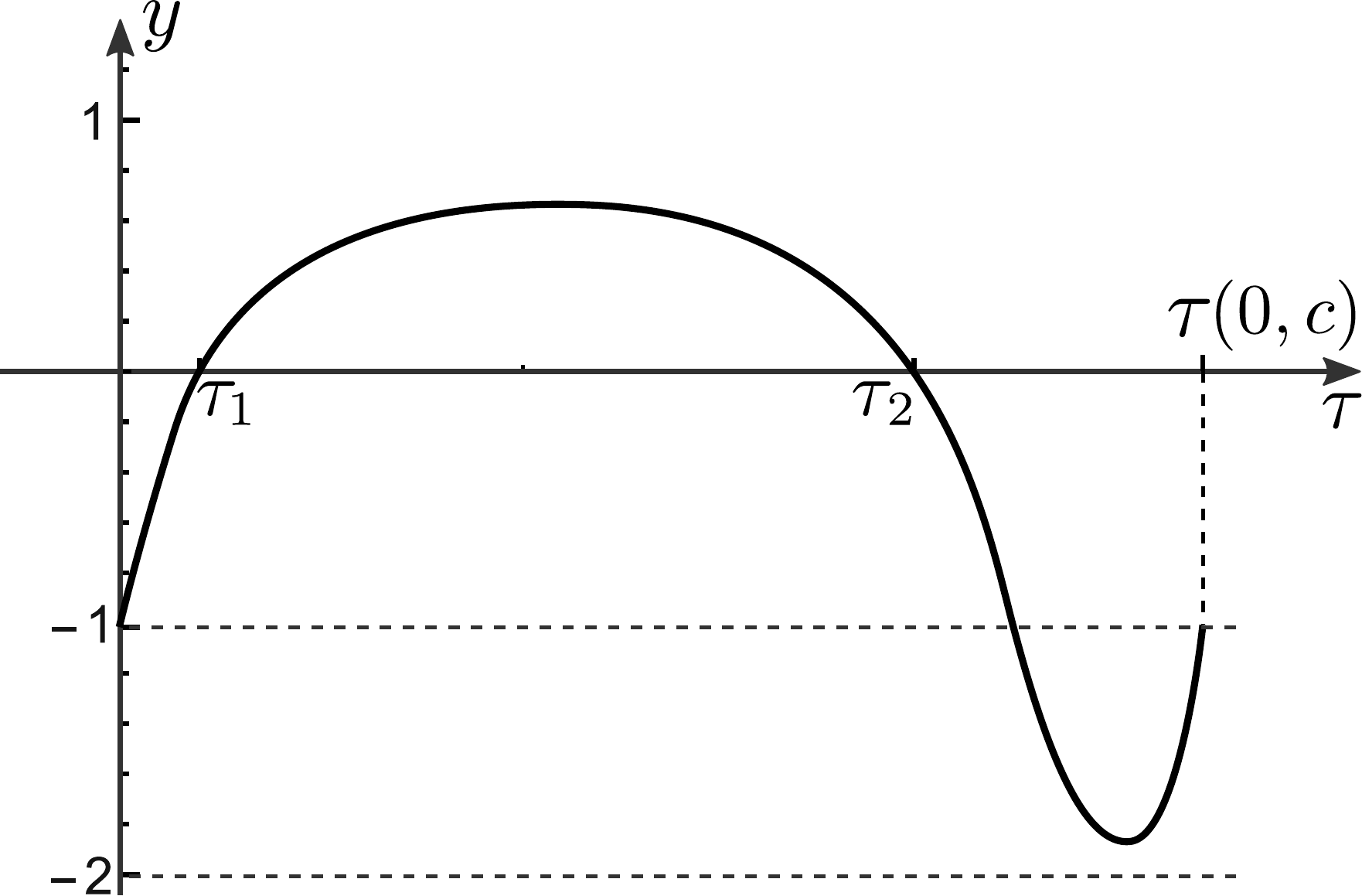} \\
$a)$ & \hspace{0.6cm} & $b)$
\end{tabular}
\caption{\label{Fig2} Graph of $y_c(\tau)-1$ for initial condition $(0,c)\in \Sigma^{+}$, with $c \in (0, 2/3)$ in $a)$ and with $c$ close to $4/3$ in $b)$. The $\tau_1$ and $\tau_2$ are the positive times that $\varphi_{\tau}(0,c)$ needs to reach the vertical line $\{y=1\}$ by first and second time, respectively.}
\end{figure}
\end{center}
For $2/3 < c < 4/3,$ we rewrite the displacement function as 
$$L(0,c) = \displaystyle \int_{\gamma_c} y-1.$$
Following \cite[p.269]{HSD}, we divide the curve $\gamma_c$ into four curves $\gamma^1_c, \gamma^2_c, \gamma^3_c, \gamma^4_c$ as shown in Figure~\ref{Fig3}. Let $y_0:= 1- \sqrt{3}$ be the intersection of $\gamma_{2/3}$ with the negative $y-$axis. 
\begin{center}
\begin{figure}[h]
\begin{tabular}{ccc}
\includegraphics[scale=0.34]{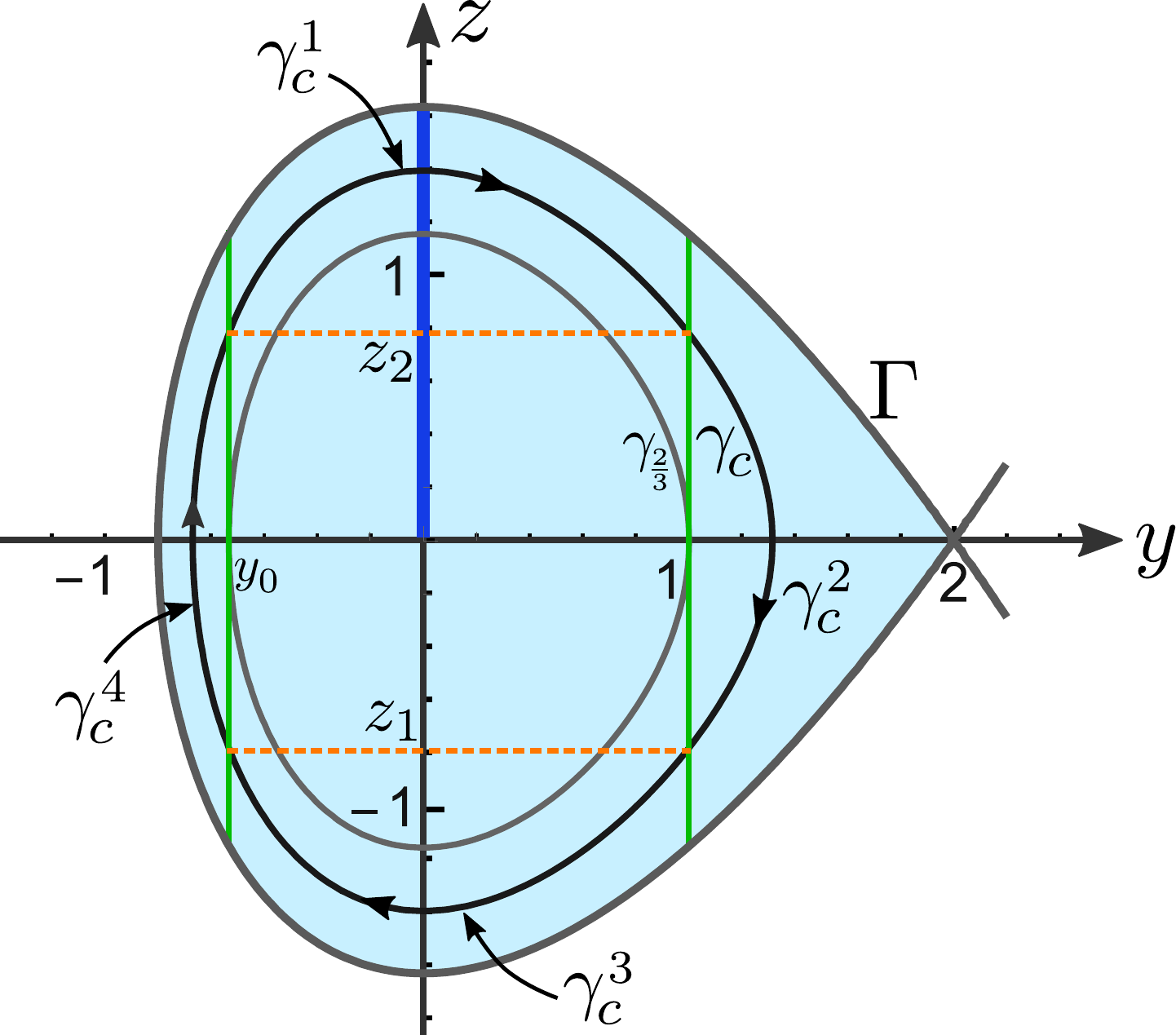}
& \hspace{0.0cm} &
\includegraphics[scale=0.33]{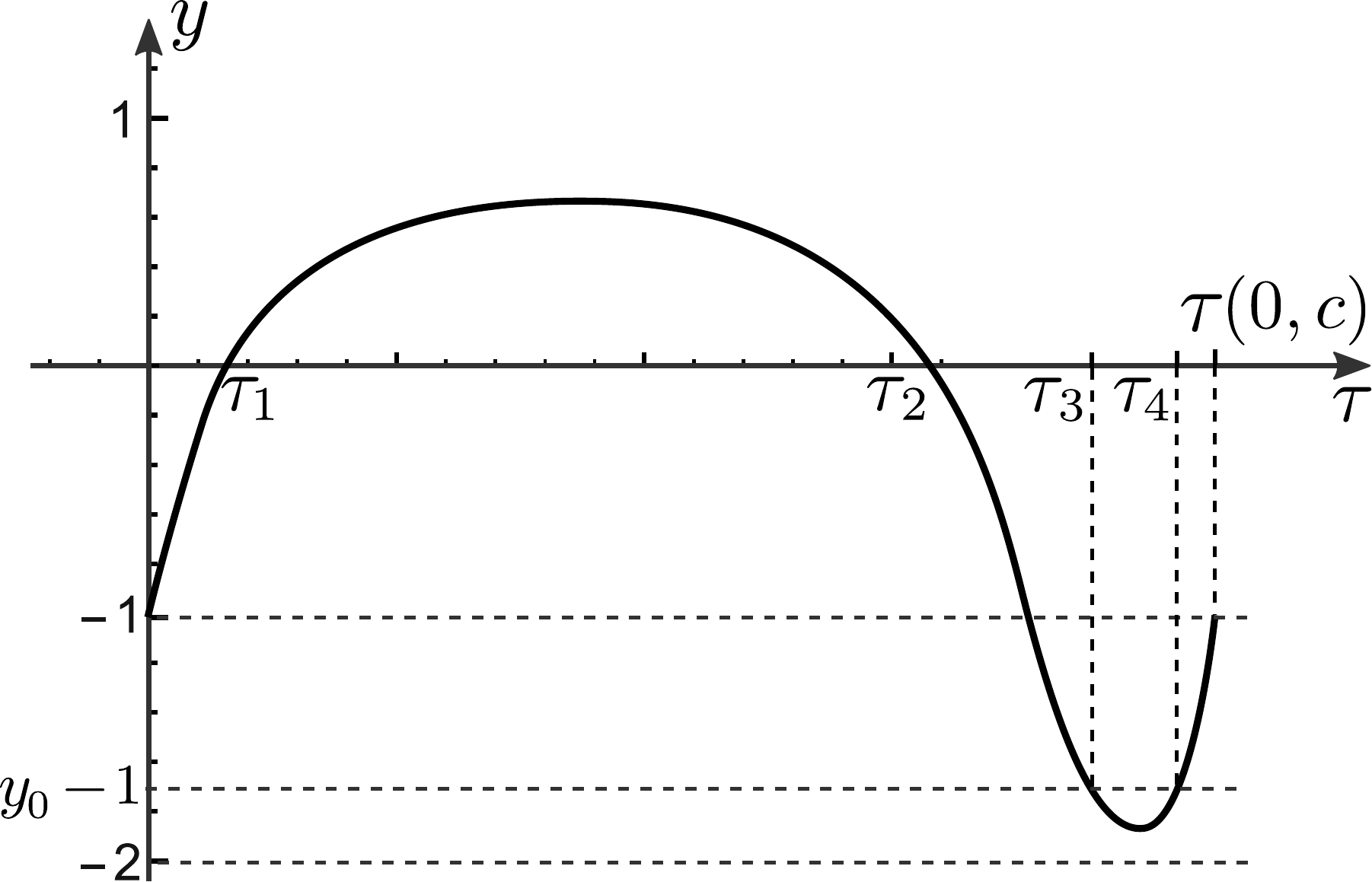} \\
$a)$ & \hspace{0.6cm} & $b)$
\end{tabular}
\caption{\label{Fig3}  The curve $\gamma_c$ divided in the curves $\gamma_c^1$, $\gamma_c^2$, $\gamma_c^3$ and $\gamma_c^4$, with $2/3<c<4/3$, is showed in $a)$. The corresponding graph of $y_c(\tau)-1$  is showed in $b)$.}
\end{figure}
\end{center}
The curves $\gamma^1_c$ and $\gamma^3_c$ are defined for $y_0 \leq y \leq 1,$ while the curves $\gamma^2_c$ and $\gamma^2_c$ are defined for $z_1 \leq z \leq z_2$ (of course, $z_1$ and $z_2$ depend on $c$). Then
$$L(0,c) = L_1(c)+L_2(c)+L_3(c)+L_4(c),$$
where
$$L_i(c):= \displaystyle \int_{\gamma^i_{c}} y-1, \,\, i=1,2,3,4.$$

We can make an analogous analysis  as in \cite[p.270]{HSD} to prove the following properties. $L_1(c)$ and $L_3(c)$ are negative monotonous increasing functions which are bounded in the interval $(2/3, 4/3)$; $L_2(c)$ is a positive monotonous increasing function, which goes to infinity as $c$ goes to $4/3$; $L_4(c)$ is a negative decreasing function, which is bounded and whose derivative goes to zero as $c$ goes to $4/3.$ The properties on $L_2(c)$ and $L_4(c)$ imply that $L_2(c) + L_4(c)$ has a unique zero and $L_2(c) + L_4(c)$ goes to infinity as $c$ goes to $4/3.$ In addition, the properties on $L_1(c)$ and $L_3(c)$ imply that $L(0,c)$ has a unique zero $c^*$ in $(2/3, 4/3).$ Furthermore, a more accurate analysis proves that $L(0,c)$ is a monotonous increasing function in $(2/3, 4/3).$
\end{proof}

Figure \ref{Fig4} shows part of the phase portrait of \eqref{sistema-diferencial-transformado-3d-lineal-a12-nocero-redu3} close to the isochronous surface $\mathcal{S}_{c^{*}}$, where $c^{*}\in (1.6305,1.6310)$.  More precisely, it gives part of the level surfaces of the first integral and the trajectories with initial conditions $(0,0,1.6323)$, in magenta,  $(0,0,1.54919)$, in red, and $(0,0,1.6308)$, $(2.7,0,1.6308)$, $(-4,0,1.6308)$ in blue. The magenta trajectory advances in the positive direction of the $x$-axis, the red trajectory advances in the negative direction of the $x$-axis and the blue ones are (approximately) periodic, with the same period.
\begin{center}
\begin{figure}[ht]
\includegraphics[scale=0.8]{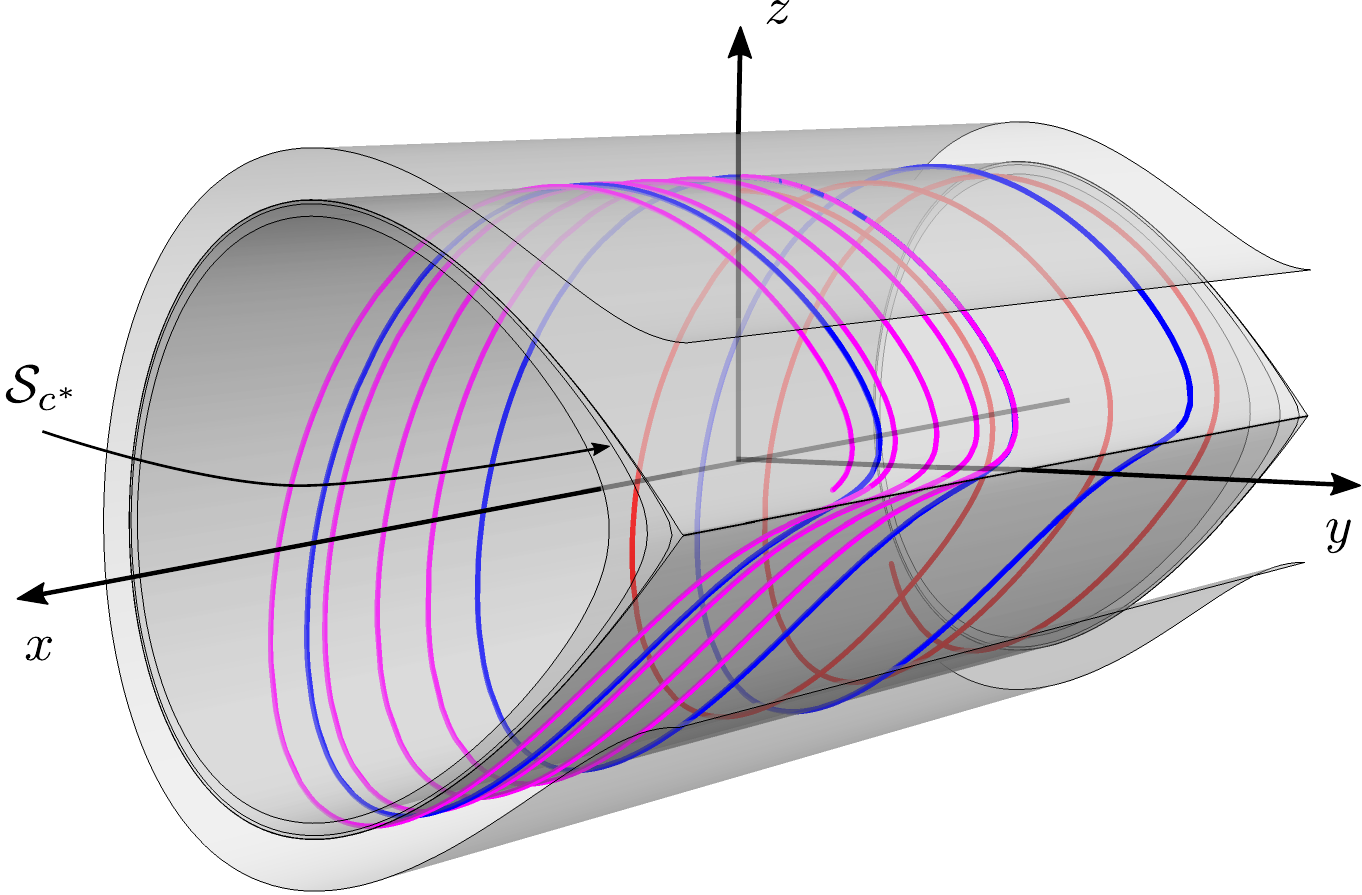} 
\caption{\label{Fig4} Existence of a isochronous surface $\mathcal{S}_{c^{*}}$ of system  \eqref{sistema-diferencial-transformado-3d-lineal-a12-nocero-redu3}.} 
\end{figure}
\end{center}

\section{Concluding remarks}
From the previous sections
we can observe that 
the discrete and 
continuous dynamics 
of the nilpotent polynomial vector fields
\eqref{campo-vectorial-3d-corto}
share some similarities. 
For instance,
from \eqref{sistema-discreto-3d-conjugado} 
it follows that 
the surface
$$
w=\frac{a_{12}}{d_2 p_{d_2}} u^2 + A_3 + \frac{a_{20}}{d_2 p_{d_2}}
$$ 
is reached  
from any initial point
after one iteration. Thus, 
this surface contains the long-term dynamics
of the system $(\R^3,\N_0,G)$, 
which is conjugated to \eqref{sistema-discreto-3d}.
Similarly, 
since system \eqref{sistema-diferencial-3d}
is polynomially integrable, its dynamics
evolves in the algebraic level surfaces of the 
polynomial first integral. 
Hence, this reduction 
of the dynamics 
in one dimension 
is a similarity that, in general, share 
the discrete and continuous dynamical systems previously mentioned.
Other similarity is that 
for the case $\deg A_1(x)=1$,
the discrete dynamical system 
\eqref{sistema-discreto-3d} 
and 
the continuous dynamical system 
\eqref{sistema-diferencial-3d} 
are completely understood.  
Indeed,
from Theorem \ref{Teorema-1}
the discrete system 
has a global attractor 
and 
from Theorem \ref{Teorema-2} 
the continuous system 
is polynomially completely integrable.

From conditions 
\eqref{cond-Pols-P-3d} 
we know that 
$\deg A_1(x)=1$ 
if $d_2 > 1$.
Hence, 
according to the previous paragraph,
the discrete and 
continuous dynamics 
of 
\eqref{campo-vectorial-3d-corto} will be more interesting for 
$d_1\geq 1$, $d_2=1$ and $\deg A_1(x)=2$. 

The system \eqref{sistema-discreto-3d} with $d_2=1$  has been studied in Statement 2) of Theorem~\ref{Teorema-1}. We have showed that in such case there exist a unique fixed point, there are not $2$-cycles and there exists at least one $3$-cycle.
This triggers the following questions:
\begin{itemize}
\item How to discern analytically the existence of $m$-cycles with  $ m \geq 4$?
\item Could be the $3$-cycle of Theorem \ref{Teorema-1} unique and attractor?
\end{itemize}     

The system \eqref{sistema-diferencial-3d}, 
with  $d_1=d_2=1$ have been analyzed in
Proposition \ref{resultados-caso-d1=d2=1}. We showed 
that when
the associated planar Hamiltonian system 
has only one period annulus, the system has
only one isochronous periodic  
surface $\mathcal{S}_{\mu}$. 
Concerning this result 
a natural question arise: 
\begin{itemize}
\item Is any periodic orbit in $\mathcal{S}_{\mu}$ persisting under
the perturbation $\lambda I$ with $\lambda < 0$?
\end{itemize}
  
A positive answer to the this question would 
give a affirmative response to the Problem 19 in \cite{GasullProblemas}, which is related with the Markus--Yamabe Conjecture.

We note that
for $d_1>1$ the 
planar Hamiltonian system associated with 
system \eqref{sistema-diferencial-3d} 
can have several period annuli. 
For instance, 
by taking 
$P_1(s)=s^2-s-3$, $P_2(s)=s$, $a_{12}=1$ and $A_3=-6$,
the system \eqref{sistema-diferencial-3d} has two period annuli. 
Hence, we can ask:
\begin{itemize}
\item How many 
periodic surfaces can have 
system \eqref{sistema-diferencial-3d} for $d_1\geq 1$ and $d_2=1$? 
\end{itemize}
This question is in some sense analogous to the problem about the number of limit cycles in planar polynomial vector fields. See \cite{Ily02,RP2009} and references there in.

In this work, we have focused in the discrete and continuous dynamics of the nilpotent polynomial vector fields in dimension three. However, we believe that the techniques used in the present research are useful also for an analogous study in higher dimensions. Recall that in \cite{CaVa} is provided the characterization of a wide class of nilpotent polynomial vector fields in any dimension. Of course, the generalization or extension of the results presented here is no simple. For example, in dimension four there are six different families of  nilpotent polynomial vector fields to be analyzed. We expect that in some of these families could be arise different behaviors than those obtained here.










\end{document}